\documentclass[11pt]{amsart}
\usepackage {enumerate}
\usepackage{setspace}
\usepackage{caption}
\usepackage{mathrsfs}
\usepackage{hyperref}
\usepackage{esint}
\usepackage{amssymb}
\usepackage{graphicx}
\usepackage{epstopdf}
\usepackage{amsthm}

\usepackage{color}
\usepackage{lipsum}

\newtheorem{theorem}{Theorem}[section]
\newtheorem{lemma}[theorem]{Lemma}
\newtheorem{corollary}[theorem]{Corollary}
\newtheorem{proposition}[theorem]{Proposition}

\numberwithin{equation}{section}
\newtheorem*{Conjecture}{Conjecture}

\theoremstyle {definition}

\newtheorem{remark}[theorem]{Remark}

\makeatletter
\@namedef{subjclassname@2020}{%
  \textup{2020} Mathematics Subject Classification}
\makeatother

\DeclareMathOperator{\Div}{div}

\DeclareMathOperator{\lip}{Lip}
\DeclareMathOperator{\dist}{dist}

\DeclareMathOperator{\Osc}{osc}
\DeclareMathOperator{\R}{\mathbf R}
\begin{document}
\title[Stability of Bernstein type theorem]{Stability of Bernstein type theorem for the minimal surface equation}
\begin{abstract}
{Let $ \Omega \subsetneq \mathbf{R}^n\,(n\geq 2)$ be an unbounded convex domain. We study the minimal surface equation in $\Omega$ with boundary value given by the sum of a linear function and a bounded uniformly continuous function in $ \mathbf{R}^n$.  If $ \Omega $ is not a half space, we prove that the solution is unique. If $ \Omega $ is a half space, we prove that graphs of all solutions form a foliation of $\Omega\times\mathbf{R}$. This can be viewed as a stability type theorem for Edelen-Wang's Bernstein type theorem in \cite{EW2021}. We also establish a comparison principle for the minimal surface equation in $\Omega$.}
\end{abstract}
%We obtain a stability type theorem for Edelen-Wang's Bernstein type theorem in \cite{EW2021}. Let $ \Omega \subsetneq \mathbf{R}^n\,(n\geq 2)$ be an unbounded convex domain. We study the minimal surface equation in $\Omega$ with boundary value given by the sum of a linear function and a bounded uniformly continuous function in $ \mathbf{R}^n$.  If $ \Omega $ is not a half space, we prove that there is a unique solution of the minimal surface equation whose boundary value is given by the sum of a linear function and a bounded uniformly continuous function in $ \mathbf{R}^n$. If $ \Omega $ is a half space, we prove that all solutions form a foliation of $\Omega\times\mathbf{R}$. A comparison principle for the minimal surface equation in $\Omega$ is also established.}

\author{Guosheng Jiang}
\address{School of Mathematical Sciences, Laboratory of Mathematics and Complex Systems, MOE, Beijing Normal University, Beijing, 100875, China}
\email{gsjiang@bnu.edu.cn}
\author{Zhehui Wang}
\address{Academy of Mathematics and Systems Science, Chinese Academy of Sciences, Beijing, 100190, China}
\email{wangzhehui@amss.ac.cn}
\author{Jintian Zhu}
\address{Beijing International Center for Mathematical Research, Peking University, Beijing, 100871, China}
\email{zhujintian@bicmr.pku.edu.cn}
\maketitle
\section{Introduction}
In the research of minimal graphs over Euclidean space, one important result is the Bernstein theorem, which says that entire minimal graphs over $\mathbf R^n$ with $2\leq n\leq 7$ are hyperplanes (refer to \cite{Bernstein27, DeGiorgi65, Almgren66, Simons68}). For $n\geq 8$, there were non-planar entire minimal graphs constructed by Bombieri, De Giorgi and Giusti \cite{BDGG69} and so the Bernstein theorem fails in higher dimensions. Just recently, Edelen and the second named author \cite{EW2021} established a Bernstein type theorem for minimal graphs over unbounded convex domains. In {\it any} dimension they can show that if the boundary of a minimal graph over an unbounded convex domain $\Omega\subsetneq \mathbf{R}^n$ is contained in a hyperplane, then so is the minimal graph itself.

Here we would like to interpret the Edelen-Wang's theorem from the view of partial differential equations, and the situation can be divided into two cases:
\begin{itemize}
\item If the unbounded convex domain $\Omega\subsetneq \mathbf{R}^n$ is not a half space, the Edelen-Wang's theorem can be understood as the {\it existence} and {\it uniqueness} for solutions of the minimal surface equation in $\Omega$ with a linear boundary value. 
\item If the unbounded convex domain $\Omega\subsetneq \mathbf{R}^n$ is a half space, the uniqueness fails but all solutions form a one-parameter family, whose graphs form a {\it foliation} of $\Omega\times\mathbf{R}$.
\end{itemize}

This interpretation motivates us to further study the stability of the Edelen-Wang's theorem. We are going to show that the existence, uniqueness and foliation structure are preserved even if the linear boundary value is perturbed by a bounded and uniformly continuous function.

In order to state our main theorem, we introduce some necessary notation. In the rest of this paper, let $n\geq 2$ and $\Omega\subsetneq \mathbf{R}^n$ be an unbounded domain. Moreover, let $l:\mathbf R^n\to \mathbf R$ be a linear function and $\phi:\mathbf R^n\to \mathbf R$ a bounded and uniformly continuous function. In the following, we consider the minimal surface equation
\begin{equation}\label{Eq: MSE}
\mathcal Mu:=\Div\left(\frac{\nabla u}{\sqrt{1+|\nabla u|^2}}\right)=0\mbox{ in } \Omega
\end{equation}
with the Dirichlet boundary value
\begin{equation}\label{Eq: boundary value}
u=l+\phi \mbox{ on } \partial\Omega.
\end{equation}

Now our main theorem can be stated as following.
\begin{theorem}\label{Thm: main}
The following statements are true:
\begin{itemize}
\item[1.] If $\Omega\subsetneq \mathbf{R}^n$ is a convex domain but not a half space, then \eqref{Eq: MSE}-\eqref{Eq: boundary value} has a unique solution.
\item[2.] If $\Omega\subsetneq \mathbf{R}^n$ is a half space, then all solutions of \eqref{Eq: MSE}-\eqref{Eq: boundary value} form a one-parameter family and the graphs of these solutions form a foliation of $\Omega\times \mathbf R$.
\end{itemize}
\end{theorem}

When $\Omega$ is a convex domain but not a half space, the existence comes from the standard exhaustion argument while the uniqueness part turns out to be much more difficult. 
{Before our work, the discussions on the uniqueness mainly focus on dimension two. In particular, Nitsche \cite{Nit65} came up with the following conjecture.
\begin{Conjecture}
Suppose $D\subset \R^2$ is contained in a wedge with opening angle less than $\pi$. Then, the solution of the minimal surface equation with continuous boundary value is unique.
\end{Conjecture}
It turns out that the original Nitsche's conjecture is too ideal. Actually, Collin \cite{Collin90} provided a counterexample indicating that the uniqueness can not hold if the boundary value grows too fast. On the other hand, after further requiring the boundary value to be bounded,  Hwang \cite{Hw88} and M\={\i}kljukov \cite{VM79} confirmed Nitsche's conjecture independently. Moreover, when $D\subset \R^2$ is the union of a compact convex subset with finitely many disjoint half strips attached to its boundary, Sa Earp and Rosenberg \cite{SaRo89} proved the uniqueness of the solution of the minimal surface equation with bounded uniformly continuous boundary value. 
%Collin and Krust \cite{CK91} proved that two solutions $u$ and $v$ of the minimal surface equation over unbounded domain $\tilde\Omega$ with the same { continuous} boundary value must coincide if
%$$
%\max_{\tilde\Omega\cap \overline{B_R}}|u-v|=o\left(\int_{R_0}^R\frac{\mathrm dr}{|\Omega\cap\partial B_r|}\right),\mbox{ as } R\to+\infty.
%$$
%The order on the right hand side becomes $o(\log R)$ for general unbounded domains and $o(R)$ for stripes, and this is optimal concerning the example from \cite{Collin90}. The proof of Collin and Krust is based on the Weierstrass representation of minimal surfaces in $\mathbf R^3$, which {seems only applicable} in dimension two. 
%There are other researches (for example \cite{CK91, VM79}) on the uniqueness for solutions of the minimal surface equation over unbounded domains, 

For higher dimensions, Massari and Miranda \cite{MM84} showed the existence of a solution of the minimal surface equation in unbounded convex domain with any continuous boundary value, but the uniqueness problem remained unsolved. Again Collin's counterexample suggests that the uniqueness needs to be considered with boundary values satisfying controlled growth (e.g. bounded). With a further limit to bounded uniformly continuous boundary values, our work gives a partial answer to Massari-Miranda's uniqueness problem. Namely, we have the following immediate corollary of Theorem \ref{Thm: main}.}
\begin{corollary}\label{corunique}
If $\Omega\subsetneq \mathbf{R}^n$ is a convex domain but not a half space, then the solution of the minimal surface equation in $\Omega$ with bounded and uniformly continuous boundary value is unique. 
\end{corollary}
In our proof, the uniqueness comes from the following more general comparison theorem for the minimal surface equation.
\begin{theorem}\label{Thm: comparison}
Assume that $\Omega\subsetneq \mathbf{R}^n$ is a convex domain but not a half space and that $u_1$ and $u_2$ are two solutions of equation \eqref{Eq: MSE} satisfying the boundary value
\begin{equation}
u_i=l_i+\phi_i\mbox{ on }\partial\Omega,\,\, i=1, 2,
\end{equation}
where $\phi_1$ and $\phi_2$  are bounded and uniformly continuous functions. If $u_1\leq u_2$ on $\partial\Omega$, then $u_1\leq u_2$ in $\Omega$.
\end{theorem}

Based on the work in \cite{EW2021}, we can reduce above theorem to the following special one. {We say that a domain $\Omega$ satisfies the {\it exterior cone property} if there is an infinity cone outside $\Omega$.}
\begin{proposition}\label{Prop: almost linear comparison theorem}
Let $\Omega$ be an unbounded domain with the exterior cone property. Assume that $u_1$ and $u_2$ are two solutions of equation \eqref{Eq: MSE} satisfying
\begin{equation}
u_i=l+\phi_i \mbox{ on } \partial\Omega,\,\, i=1,2,
\end{equation}
and
\begin{equation}\label{pro1.3}
\|u_i-l\|_{C^0(\bar\Omega)}<+\infty,\,\, i=1,2,
\end{equation}
where $\phi_1$ and $\phi_2$ are bounded and uniformly continuous functions with $\phi_1\leq \phi_2$. Then $u_1\leq u_2$.
\end{proposition}

When $\Omega$ is a half space, we can construct a family of solutions of \eqref{Eq: MSE}-\eqref{Eq: boundary value} characterized by their growth rate at infinity. In fact, the graph of such a solution has a unique ``approximate hyperplane'' in the form of $\{x: l(x)+cx_n=0\}$ (refer to Proposition \ref{existhalf}). On the other hand, given any solution of \eqref{Eq: MSE}-\eqref{Eq: boundary value}, we can prove that it must belong to the family of solutions in our construction (refer to Corollary \ref{halfconstcoro}). Hence $c$ can serve as a parametrization for all solutions, {and we denote by $u_c$ the solution having ``approximate hyperplane'' in the form of $\{x: l(x)+cx_n=0\}$. The precise meaning of ``foliation" in Theorem \ref{Thm: main} is that the map
$$
\Phi:\mathbf R^n_+\times \mathbf R\to \mathbf R^n_+\times \mathbf R,\,\, (x,c)\mapsto (x,u_c(x))
$$
can be shown to be a homeomorphism (refer to Proposition \ref{homeo}). Moreover, the homeomorphism can be improved to be a $C^1$-diffeomorphism after changing parametrization if $\phi$ is $C^1$ with $\|\phi\|_{C^1(\partial\mathbf R_+^n)}<\infty$ (refer to Proposition \ref{halfc1}). %on $\partial\mathbf R_+^n$

The rest of this paper is organized as follows. Section \ref{existence} contains the proof of the existence of \eqref{Eq: MSE}-\eqref{Eq: boundary value} when $\Omega\subsetneq \mathbf{R}^n$ is an unbounded convex domain. In Section \ref{decayblowup}, we present a decay or blow up alternative for linear elliptic equations of divergence form, {which will be used later}. In Section \ref{unique}, we prove the first part of Theorem \ref{Thm: main}, while the second part is proved in Section \ref{forliation}.

\section{Existence: the exhaustion method}\label{existence}
\subsection{Construction of exhaustion domains}
Here we devote to show the following proposition.

\begin{proposition}\label{exhaust}
If $\Omega\subsetneq\mathbf{R}^n$ is an unbounded convex domain, then there is an exhaustion $\{\Omega_k\}_{k\geq 1}$ of $\Omega$ satisfying the following
\begin{itemize}
\item $\Omega_k$ is smooth and convex;
\item $\Omega_k\subset B_{k}\cap \Omega$;
\item $\Omega_k\subset \Omega_{k+1}.$
\end{itemize}
\end{proposition}
\begin{proof}
Set $U_k=B_k\cap \Omega$ and $d_k(x):=\dist(x,\partial U_k)$. {We note that $d_k\leq d_{k+1}$, $U_k$ is a convex domain, and $d_k$ is a convex function.} By \cite[Theorem 1]{AD}, there is a smooth convex function $f_k: U_k\to \mathbf{R}$, such that for any $x\in U_k$, $$d_k(x)-\frac{1}{k^2}\leq f_k(x) \leq d_k(x).$$ 
Let $$\Omega_k:=\left\{x\in U_k: f_k(x)>\frac{1}{k}\right\}.$$We note $\Omega_k$ is a convex subset of $U_k$, and without loss of generality we can assume it is also smooth by Sard's theorem. Let us show $\{\Omega_k\}_{k\geq 1}$ forms a exhaustion of $\Omega$. For any $x\in\Omega$, there is a $k$ large enough such that $x\in U_k$ and $d_k(x)>2/k$. Then we have $f_k(x)>1/k$ and $x\in\Omega_k$. Therefore, $\bigcup_{k\geq1} \Omega_k=\Omega$. Moreover, for any $x\in \Omega_k$, we have 
\begin{align*}f_{k+1}(x)\geq d_{k+1}(x)-\frac{1}{(k+1)^2}&\geq d_k(x)-\frac{1}{(k+1)^2}\\&\geq f_k(x)-\frac{1}{(k+1)^2}>\frac{1}{k+1}.
\end{align*}
As a consequence, $x\in \Omega_{k+1}$ and then we have $ \Omega_{k}\subset \Omega_{k+1}$.
\end{proof}

\subsection{Proof of existence} Here we present the proof of the existence of \eqref{Eq: MSE}-\eqref{Eq: boundary value}. We gives the detail for the case when $\Omega\subsetneq\mathbf{R}^n$ is an unbounded convex domain but not a half space, and we omit the half space case since it is similar.
\begin{proposition}\label{exhaustion}
If $\Omega\subsetneq\mathbf{R}^n$ is a convex domain but not a half space, then there is a solution $u$ of \eqref{Eq: MSE}-\eqref{Eq: boundary value} with
$$
\|u-l\|_{C^0(\Omega)}\leq \|\phi\|_{C^0(\mathbf{R}^n)}.
$$
\end{proposition}
\begin{proof}
Let  $\{\Omega_k\}$ be the exhaustion constructed in Proposition \ref{exhaust}, we consider Dirichlet problems:
$$\mathcal Mu_k=0 \mbox{ in } \Omega_k;\,\, u_k=l+\phi \mbox{ on } \partial \Omega_k. $$
By the maximum principle,  we have for any $k>0$ and $x\in \Omega_k$,
$$l(x)-\|\phi\|_{C^0(\mathbf{R}^n)}\leq u_k(x)\leq l(x)+\|\phi\|_{C^0(\mathbf{R}^n)}.$$
For any $\Omega'\subset\subset \Omega_k$, by the interior gradient estimate of the minimal surface equation (see \cite{GT, Han}), and $ \mathop{\mathrm {osc}}\limits_{\Omega_k}u_k\leq 2\|\phi\|_{C^0(\mathbf{R}^n)}$, we have $$\sup_{x\in\Omega'}|\nabla u_k(x)|\leq C(n, \|\phi\|_{C^0(\mathbf{R}^n)},\dist(\Omega', \partial \Omega)).$$
Fixing any $\Omega'\subset\subset \Omega$, we choose $k_0>1$ large enough such that $\Omega'\subset\subset\Omega_{k_0}$, then $\{u_k\}_{k\geq k_0}$ is well defined in $\Omega'$. By the Schauder estimate to $u_k-l$, we have $$\|u_k-l\|_{C^{2,\alpha}(\Omega')}<C(n, \|\phi\|_{C^0(\mathbf{R}^n)}, \dist(\Omega', \partial \Omega)).$$
Hence after pass to a subsequence, $\{u_k\}_{k\geq 1}$ converges to a $C^2(\Omega)$ function and we denote it by $u$. It is obvious that $\mathcal Mu=0$ in $\Omega$ and $$
\|u-l\|_{C^0(\Omega)}\leq \|\phi\|_{C^0(\mathbf{R}^n)}.
$$
Fix any $x_0\in \partial \Omega$.  For any $\varepsilon>0$, there is a constant $\delta>0$, such that, for any $x\in B_{\delta}(x_0)\cap\overline\Omega$, $$|l(x)-l(x_0)|+|\phi(x)-\phi(x_0)|<\varepsilon,$$ and we fix such a point $x$. 

Take $\{x_k\}_{k\geq 1}$ with $x_k\in \partial \Omega_k$ and $x_k\to x_0$ as $k\to\infty$. Note there is $k_1>1$ large enough such that {for all $k\geq k_1$, we have $|x_k-x_0|<\delta$, $x\in\Omega_k$ and $$|u(x)-u_k(x)|<\varepsilon.$$}By the estimate of the modulus of continuity of solutions for the minimal surface equation in bounded domain (see \cite[Theorem 3.2.3]{Han}), we can assume $$|u_k(x)-l(x_k)-\phi(x_k)|<\varepsilon.$$ Then,
\begin{align*}
|u(x)-l(x_0)-\phi(x_0)|<&|u(x)-u_k(x)|+|u_k(x)-l(x_k)-\phi(x_k)|\\&+|l(x_k)+\phi(x_k)-l(x_0)-\phi(x_0)|\\<&3\varepsilon.
\end{align*}
Hence, $\lim\limits_{x\to x_0, x\in\Omega}u(x)=l(x_0)+\phi(x_0).$
\end{proof}

With a similar argument we can show

\begin{proposition}\label{existhalf}
If $\Omega$ is the half space $\mathbf{R}_+^n:=\{x\in \mathbf{R}^n: x_n>0\}$, then for any real constant $c$ there is a solution $u_c$ of \eqref{Eq: MSE}-\eqref{Eq: boundary value} with
\begin{equation*}
\|u_c-l-cx_n\|_{C^0(\Omega)}\leq \|\phi\|_{C^0(\mathbf{R}^n)}.
\end{equation*}
\end{proposition}

\section{Linear theory: Decay or blow up alternative}\label{decayblowup}
In this section, we will consider the following equation
\begin{equation}\label{Eq: linear equation}
\mathcal Lu:=\Div(A(x)\nabla u)=0 \mbox{ in } \Omega
\end{equation}
with the condition
\begin{equation}\label{Eq: positive solution}
u=0\mbox{ on }\partial\Omega, \mbox{ and } u>0 \mbox{ in } \Omega,
\end{equation}
where
$$
{A(x)\in L^\infty(\Omega), \lambda I\leq A\leq \lambda^{-1}I,  \mbox{ and } \lambda\in(0, 1) \mbox{ is a constant.}}
$$

First we mention that \eqref{Eq: linear equation}-\eqref{Eq: positive solution} has been studied in special cases when the domain $\Omega$ is an infinite cone or an infinite cylinder (see \cite{LN85} and \cite{BWZ16} respectively). In both cases, it was proved that the dimension of the space consisting of all solutions is determined by the number of ends of the underlying domain $\Omega$. This philosophy, however, remains to be an open problem when $\Omega$ turns out to be a general unbounded domain. One cannot apply those methods from \cite{LN85, BWZ16} to general cases. Indeed, their arguments rely heavily on the scaling or translating invariance of the underlying domain $\Omega$, which guarantees that a Harnack inequality with a uniform constant even holds around the infinity. Readers can turn to our brief discussion in Appendix \ref{Sec: A} for a quick feeling.

We are more concerned with the blow up phenomenon of the solution of \eqref{Eq: linear equation}-\eqref{Eq: positive solution}, which is closely related to our work. For solutions to the equation \eqref{Eq: linear equation} in exterior domains, Moser \cite{Moser1961} established the H\"older decay or blow up alternative on the oscillation as an application of the Harnack inequality (with a uniform constant around infinity). From the same technique, such alternative can be established for solutions to \eqref{Eq: linear equation}-\eqref{Eq: positive solution} (see Lemma \ref{Lem: alternative}) by the application of boundary Harnack inequality when the domain $\Omega$ is an infinite cone or an infinite cylinder.

Here we would like to deal with more general case, where the domain $\Omega$ satisfies the {\it exterior cone property}, i.e. there is an infinite cone outside $\Omega$. Given an unbounded domain $\Omega$, we use $\Omega_{ext}$ and $(\partial\Omega)_{ext}$ to denote $\Omega-K$ and $\partial\Omega-K$ for some fixed compact set $K$. The main result in this section is the following theorem.
\begin{theorem}\label{Thm: alternative}
Let $\Omega$ be an unbounded domain satisfying the exterior cone property and $u\in W^{2,p}_{loc}(\Omega_{ext})\cap C^0(\bar\Omega_{ext})$ be a solution of the equation \eqref{Eq: linear equation} in $\Omega_{ext}$ with $u=0$ on $(\partial\Omega)_{ext}$ and $u>0$ in $\Omega_{ext}$. Then the function
$$
\mathrm{osc}(r)=\max_{\Omega\cap \partial B_r} u
$$
has a limit as $r\to+\infty$, and the limit is either $+\infty$ or $0$. 
\end{theorem}

As a preparation, we have
\begin{proposition}\label{Prop: unbounded estimate}
If $\Omega$ is an unbounded domain satisfies the exterior cone property, then any solution of \eqref{Eq: linear equation}-\eqref{Eq: positive solution} must be unbounded.
\end{proposition}
\begin{figure}[htbp]
\centering
\includegraphics[width=7cm]{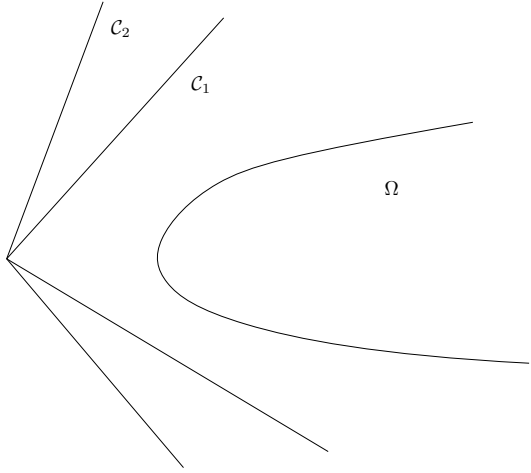}
\caption{Extend domain $\Omega$ to infinite cones}
\label{Fig: 1}
\end{figure}

\begin{proof}
As shown in Figure \ref{Fig: 1}, we can extend the domain $\Omega$ to infinite cones $\mathcal C_1$ and $\mathcal C_2$ from the exterior cone property. Let us define
$$
\bar A(x)=\left\{
\begin{array}{cc}
A(x),&x\in\Omega;\\
I,&x\in \mathcal C_2-\Omega,
\end{array}\right.
$$
and consider \eqref{Eq: linear equation}-\eqref{Eq: positive solution} in the infinite cone $\mathcal C_2$. From Theorem \ref{Thm: positive solution in cone} we can obtain a positive solution $w$ and we would like to show the following estimate
$$
m(r):=\inf_{\mathcal C_1\cap S_r}w\to +\infty,\mbox{ as } r\to+\infty,
$$
where $S_r$ is the sphere with radius $r$ centered at the pole of cone $\mathcal C_1$. Otherwise, the restriction $w|_{\partial\mathcal C_1}$ is a bounded continuous function on $\partial\mathcal C_1$. From the standard exhaustion method, it is not difficult to construct a bounded harmonic function $v$ in $\mathcal C_2-{\bar{\mathcal C}}_1$ such that $v=0$ on $\partial\mathcal C_2$ and $v=w|_{\partial\mathcal C_1}$ on $\partial\mathcal C_1$ as well as $v\leq w$. From the uniqueness and Corollary \ref{Cor: unbounded} the function $w$ is unbounded in $\mathcal C_2-\bar{\mathcal C}_1$ and so $w-v$ is a positive harmonic function in
$\mathcal C_2-\bar{\mathcal C}_1$ with vanishing boundary value. It follows from Corollary \ref{Cor: harmonic in cone} that
$$
w-v=cr^\beta\phi_1
$$
for some positive constants $c$ and $\beta$. The Harnack inequality then yields $m(r)\to +\infty$ as $r\to+\infty$, which leads to a contradiction.

Now we can use $w$ as a comparison function to deduce a contradiction under the assumption that there is a bounded solution $u$ of \eqref{Eq: linear equation}-\eqref{Eq: positive solution}. From the maximum principle it is easy to see $u\leq \epsilon w$ for any $\epsilon >0$. However, this implies $u\equiv 0$ which is impossible.
\end{proof}

Now let us prove Theorem \ref{Thm: alternative}.

\begin{proof}[Proof for Theorem \ref{Thm: alternative}]
By the maximum principle, we see that the function $\Osc(r)$ is monotone when $r$ is large enough. There are two possibilities:

{\it Case 1. $\Osc(r)$ is monotone increasing when $r\geq r_0$.} Let us deuce a contradiction when $\Osc(r)$ is uniformly bounded from above. As in the proof of Proposition \ref{Prop: unbounded estimate}, we can construct a comparison function $w$ such that
$$
u\leq \Osc(r_0)+\epsilon w \mbox{ in }\Omega-B_{r_0}
$$
for any $\epsilon>0$. As a result, $\Osc(r)\equiv \Osc(r_0)$ when $r\geq r_0$. The strong maximum principle yields that $u$ is a constant function, which is impossible. So we have $\Osc(r)\to +\infty$ as $r\to+\infty$.

{\it Case 2. $\Osc(r)$ is monotone decreasing when $r\geq r_0$.} In this case, we still adopt the contradiction argument and assume that $\Osc(r)$ converges to a positive constant as $r$ tends to infinity. Extend the domain $\Omega$ to a infinite cone $\mathcal C$ and define
$$
\bar A(x)=\left\{
\begin{array}{cc}
A(x),&x\in\Omega;\\
I,&x\in \mathcal C-\Omega.
\end{array}\right.
$$
Let $\phi$ be a nonnegative function on $\partial\mathcal C$ with compact support which is positive somewhere. Then we can solve a bounded positive function $v$ such that
\begin{equation*}
\Div(\bar A(x)\nabla v)=0 \mbox{ in } \mathcal C, \mbox{ and } v=\phi \mbox{ on } \partial \mathcal C.
\end{equation*}
It follows from the decay or blow up alternative (Lemma \ref{Lem: alternative}) that
$$
\lim_{|x|\to +\infty} v(x)=0.
$$
Up to a scaling we can assume $u<v$ on $\partial(\Omega-B_{r_0})$. Let us consider the function $w=u-v$ in the domain
$$
\Omega_+:=\{u-v>0\}\subset \Omega.
$$
Note $\Omega_+$ must be non-empty and unbounded since $\Osc(r)$ has a positive limit. And $w$ is clearly a bounded positive function satisfying \eqref{Eq: linear equation}-\eqref{Eq: positive solution} in $\Omega_+$. 
However, since $\Omega_+$ satisfies the exterior cone property we know $w$ must be unbounded by Proposition \ref{Prop: unbounded estimate}. This is a contradiction and we complete the proof.
\end{proof}

\section{Uniqueness}\label{unique}
We will apply the following uniform continuity lemma and Proposition \ref{Prop: unbounded estimate} to prove Proposition \ref{Prop: almost linear comparison theorem}.
\begin{lemma}\label{uni-continuous}
If $\Omega$ is a convex domain and $u$ is a solution of \eqref{Eq: MSE}-\eqref{Eq: boundary value} satisfying
$$
\|u-l\|_{C^0(\Omega)}<+\infty,
$$ then $u$ is uniformly continuous in $\bar\Omega$.
\end{lemma}
\begin{proof}
For any $\tau>0$, set $$\eta_\tau=\frac{\tau\lip l+\|u-l\|_{C^0(\Omega)}+\|\phi\|_{C^0(\Omega)}}{\tau^2}|x|^2,$$ where $\lip l$ is the Lipschitz constant of $l$. It is clear that $\eta_\tau$ is a subsolution of \eqref{Eq: MSE}.
Let us construct a suitable comparison function $w$ by solving the equation $\mathcal Mw_\tau=0$ in $B^+_1:=B_1\cap \mathbf R^n_+$ with the Dirichlet boundary value $w_\tau=\eta_\tau$ on $\partial B_1^+$. (Actually, we need to modify $B_1^+$ to a smooth mean-convex domain but this is not much more difficult and it will not affect our argument later.) By the maximum principle, it is easy to check $$w_\tau \geq\eta_\tau \mbox{ in }B_1^+.$$

Fix any point $x$ on $\partial\Omega$ and let $P$ be a supporting hyperplane at the point $x$. After translation and rotation, we can assume that $x$ is the origin and $P$ is the hyperplane $\{x_n=0\}$ without loss of generality. Since $\phi$ and $l$ are both uniformly continuous, for any $\epsilon>0$ there is a $\delta>0$ such that 
$$
|(l+\phi)(y)-(l+\phi)(x)|\leq \epsilon,{ \mbox{ for all } y\in \partial\Omega\mbox{ satisfying }|y-x|\leq \delta.}
$$
Consider the function 
$$
v_+=(l+\phi)(x)+\epsilon+w_\delta,
$$
Clearly, $v_+$ solves the minimal surface equation and $u\leq v_+$ on $\partial(B_\delta\cap\Omega)$. Therefore, it follows from the maximum principle that $u\leq v_+$ in $B_\delta\cap\Omega$. Let
$$
v_-=(l+\phi)(x)-\epsilon-w_\delta.
$$
The same argument leads to the fact $u\geq v_-$ in $B_\delta\cap\Omega$. Now we can take $\delta'$ small enough such that 
$$
|u(y)-(l+\phi)(x)|\leq 2\epsilon, \mbox{ for all }\,y\in B_{\delta'}\cap\bar\Omega.
$$
This completes the proof.
\end{proof}

\begin{proof}[Proof of Proposition \ref{Prop: almost linear comparison theorem}]
Assume $u_1>u_2$ at some point $x_0$. Set 
\begin{align*}
&\epsilon=(u_1(x_0)-u_2(x_0))/2,\\ 
&w=(u_1-u_2-\epsilon), \mbox{ and }\Omega_\epsilon=\{w>0\}.
\end{align*} By Lemma \ref{uni-continuous} we know $\dist(\Omega_\epsilon, \partial\Omega)$ has a positive lower bound and we denote it by $\delta$. For any $x'$ in $\Omega_\epsilon$, we apply the interior gradient estimate for the minimal surface equation (see \cite[Theorem 16.5]{GT}) to $u_i\, (i=1, 2)$  in $B_{\frac{\delta}{2}}(x')$, and it yields 
$$|\nabla u_i(x')|\leq \exp\left(C\frac{\Osc_{B_{\delta/2}}u_i}{\delta}\right)\leq \exp\left(C\frac{2\|u_i-l\|_{C^0({\Omega})}+\delta \lip l}{\delta}\right),$$
where $C$ is a uniform constant depending only on $n$. It follows that $|\nabla u_1|$ and $|\nabla u_2|$ are uniformly bounded in $\Omega_\epsilon$. Therefore, $w$ satisfies a uniformly elliptic equation in $\Omega_\epsilon$ and vanishes on $\partial\Omega_\epsilon$. We also note that $w$ is bounded due to \eqref{pro1.3}. If $\Omega_\epsilon$ is bounded, then the maximum principle shows $w=0$ and it is a contradiction. If $\Omega_\epsilon$ is unbounded, note that it satisfies the exterior cone condition, by Proposition \ref{uni-continuous} $w$ can not be bounded and it is a contradiction again.
\end{proof}

We point out that Edelen-Wang's argument in \cite{EW2021} leads to the following comparison theorem.
\begin{proposition}
If $\Omega$ is a convex domain but not a half space and $u$ is a solution of \eqref{Eq: MSE} with $u\leq l$ on $\partial\Omega$, then $u\leq l$ in $\Omega$.
\end{proposition}
\begin{proof}
Suppose $X:=\{x\in \Omega: u(x)-l(x)>0\}\neq\emptyset$ and $Y\subset X$ be one of its connected component. Then $u|_{Y}: Y\to \mathbf{R}$ is a solution of the minimal surface equation with Dirichlet boundary value $l$. Since $Y$ is contained in a convex cone or a slab, then by Edelen and Wang \cite{EW2021} (actually their arguments still work for domains contained in slabs), $u|_Y=l$, which is a contradiction. Hence, $X=\emptyset$.
\end{proof}
This leads to the following quick corollary.
\begin{corollary}
If $\Omega\subsetneq\mathbf{R}^n$ is a convex domain but not a half space and $u$ is a solution of \eqref{Eq: MSE}-\eqref{Eq: boundary value}, then
\begin{equation}\label{Eq: C0 estimate}
l-\|\phi\|_{C^0(\partial\Omega)}\leq u \leq l+\|\phi\|_{C^0(\partial\Omega)}.
\end{equation}
\end{corollary}
Now we are ready to prove Theorem \ref{Thm: comparison}.
\begin{proof}[Proof for Theorem \ref{Thm: comparison}]
{Without loss of generality, we assume $l_1(0)=l_2(0)=0.$}
When $l_1=l_2$, the desired result follows directly from Proposition \ref{Prop: almost linear comparison theorem} and estimate \eqref{Eq: C0 estimate}. If $l_1\neq l_2$, we may write $l_i=l'+c_ix_n$ for some linear function $l':\mathbf R^{n-1}\to \mathbf R$ and constants $c_1>c_2$ by choosing a suitable coordinate of $\mathbf R^n$. Then we divide the discussion into two cases:

{\it Case 1.} The domain $\Omega$ has all supporting plane parallel to the hyperplane $P=\{x_n=0\}$. Since $\Omega$ is not a half space, it must be a stripe parallel to $P$. Then points in $\Omega$ has a uniform bound in $x_n$ and so we have
$$
\|u_i-l'\|_{C^0(\Omega)}<+\infty.
$$
Clearly $u_i=l'+\phi_i'$ on $\partial\Omega$, where $\phi_i'=\phi_i+c_i x_n$ keeps bounded and uniformly continuous. This reduces to the case when $l_1=l_2$.

{\it Case 2.}
The domain $\Omega$ has one supporting plane $\Sigma$ not parallel to $P$.

Let us consider the domain 
$$
H_+=\left\{x\in\mathbf R^n:x_n>\frac{\|\phi_1\|_{C^0(\partial\Omega)}+\|\phi_2\|_{C^0(\partial\Omega)}}{c_1-c_2}\right\}.
$$
\begin{figure}[htbp]
\centering
\includegraphics[width=4.8cm]{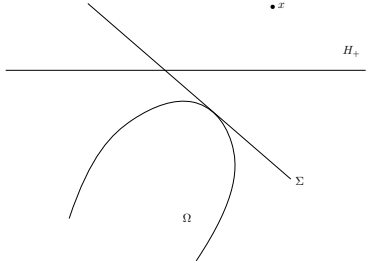}
\caption{The domains $\Omega$ and $H_+$}
\label{Fig: 3}
\end{figure}

As shown in Figure \ref{Fig: 3}, there is at least one point $x\in H_+$ outside $\Omega$ since $H_+$ crosses the hypersurface $\Sigma$.
From the fact $u_1\leq u_2$ on $\partial\Omega$ we can verify $\partial\Omega\cap H_+=\emptyset$. Combined with the convexity of $H_+$, we see that $\Omega\cap H_+$ is empty. 

On the other hand, estimate \ref{Eq: C0 estimate} yields $u_1<u_2$ in
$$
H_-=\left\{x\in\mathbf R^n:x_n<-\frac{\|\phi_1\|_{C^0(\partial\Omega)}+\|\phi_2\|_{C^0(\partial\Omega)}}{c_1-c_2}\right\}.
$$
So we just need to show $u_1\leq u_2$ in $\Omega'=\Omega-\overline{H_-}$. 
%\begin{figure}[htbp]
%\centering
%\includegraphics[width=4.8cm]{4.eps}
%\caption{The domain $\Omega'$}
%\label{Fig: 4}
%\end{figure}
Since $\partial\Omega'$ is contained in $\partial H_-$ and $\partial\Omega$, we know $u_1\leq u_2$ on $\partial\Omega'$. Notice that points in $\Omega'$ are uniformly bounded in $x_n$-direction and $u_i\,(i=1,2)$ are uniformly continuous on $\partial\Omega'$. Again we can reduce to the case when $l_1=l_2$.
\end{proof}

\begin{proof}[Proof for the first part in Theorem \ref{Thm: main}]
This case is a direct consequence of Theorem \ref{Thm: comparison}.
\end{proof}

\section{The foliation structure}\label{forliation}
We turn to the case when $\Omega$ is a half space in this section. We begin with the following two propositions.
\begin{proposition}\label{halfconst}
If $\Omega=\mathbf{R}^n_+$, then for any solution $u$ of \eqref{Eq: MSE}-\eqref{Eq: boundary value}, there is a unique real constant $c$ such that for any $x\in\mathbf{R}^n_+$,
$$
|u(x)-l(x)-cx_n|\leq \|\phi\|_{C^0(\partial\Omega)}.
$$
\end{proposition}
\begin{proof}

Let $$\varphi=\max\{\min\{u+\|\phi\|_{C^0(\partial\Omega)}, l\}, u-\|\phi\|_{C^0(\partial\Omega)}\}.$$ Then $\varphi=l$ on $\{x_n=0\}$. Similar as Proposition \ref{exhaustion}, we could construct a family of solutions $\{v_k\}$ to the minimal surface equation in $B^+_k$ with $v_k=\varphi$ on $\partial B^+_k$. By the maximum principle, we know $$u-\|\phi\|_{C^0(\partial\Omega)}\leq v_k\leq u+\|\phi\|_{C^0(\partial\Omega)} \mbox{ in } B^+_k.$$ Also, $v_k$ converges to a solution $v\in C^\infty(\Omega)\cap C^0(\overline\Omega)$ to the minimal surface equation with boundary value $l$. Note that we have  $$u-\|\phi\|_{C^0(\partial\Omega)}\leq v\leq u+\|\phi\|_{C^0(\partial\Omega)} \mbox{ in }\Omega.$$ By the Bernstein type theorem in \cite{EW2021}, we know $v(x)=l(x)+cx_n$ for some constant $c$. Hence, $$|u(x)-l(x)-cx_n|\leq \|\phi\|_{C^0(\partial\Omega)}.$$
\end{proof}

\begin{proposition}\label{abmono}
Suppose $\Omega=\mathbf{R}^n_+$ and there are two solution $u_a$ and $u_b$ of \eqref{Eq: MSE}-\eqref{Eq: boundary value} corresponding to two constant $a$ and $b$ as in Proposition \ref{halfconst} respectively. If $a> b$, then $u_a>u_b$ in $\Omega$; and if $a=b$, then $u_a=u_b$ in $\Omega$.
\end{proposition}
\begin{proof}
Note for any $x\in\Omega$, $$|u_a(x)-l(x)-ax_n|\leq \|\phi\|_{C^0(\partial\Omega)}$$ and $$|u_b(x)-l(x)-bx_n|\leq \|\phi\|_{C^0(\partial\Omega)}.$$
Then, $$u_a(x)-u_b(x)>-2\|\phi\|_{C^0(\partial\Omega)}+(a-b)x_n \mbox{ for all } x\in\Omega.$$
If $a>b$, then $u_a>u_b$ for $x_n$ large enough. Hence $$\{u_a-u_b< 0\}\subset \{0<x_n<\tau\}$$ for some constant $\tau$. 

Fix any $\varepsilon>0$ and suppose $\{u_a-u_b< -\varepsilon\}$ is not empty, then we have the following two cases:

{\it Case 1a.  $\{u_a-u_b< -\varepsilon\}$ is bounded.} In this case, there is a constant $R>0$ such that $u_a-u_b\geq 0$ on $\partial B_R^+$ and $\{u_a-u_b<  -\varepsilon\}\subset B_R^+$. Note $w=u_a-u_b$ satisfies the elliptic equation $\partial_i(a_{ij}\partial_jw)=0$ in $B_R^+$, where $$a_{ij}=\int_0^1\frac{1}{\sqrt{1+|p(t)|^2}}\left(\delta_{ij}-\frac{p_i(t)p_j(t)}{1+|p(t)|^2}\right)\mathrm{d}t \mbox{ and } p(t)=(1-t)\nabla u_b+t\nabla u_a.$$ By the maximum principle, $w\geq 0$ in $B_R^+$, which is a contradiction.

{\it Case 1b.  $\{u_a-u_b< -\varepsilon\}$ is unbounded.} In this case, $u_b-u_a-\varepsilon$ is a positive bounded solution to the elliptic equation $\partial_i(a_{ij}\partial_j w)=0$ in unbounded domain $\{u_a-u_b< -\varepsilon\}$ with zero boundary value. Note that by Lemma \ref{uni-continuous}, there is a $\delta>0$ such that $$\{u_a-u_b< -\varepsilon\}\subset\{x_n>\delta\}.$$ By the standard interior gradient estimate (as in the proof of Proposition \ref{Prop: almost linear comparison theorem}), we know $|\nabla u_a|$ and $|\nabla u_b|$ are uniformly bounded in  $\{x_n>\delta\}$. Therefore, $a_{ij}$ is uniformly elliptic. However, Proposition \ref{Prop: unbounded estimate} implies $u_b-u_a-\varepsilon$ is unbounded, which is a contradiction.

In conclusion, if $a>b$ then $\{u_a-u_b< -\varepsilon\}$ is empty for any $\varepsilon>0$. So $u_a-u_b>0$ in $\{x_n>0\}$.

On the other hand, if $a=b$, then $u_a-u_b$ is bounded. Similar as above, we fix any $\varepsilon>0$ and separate the proof to two cases.

{\it Case 2a.  $\{u_a-u_b< -\varepsilon\}$ is bounded.} In this case, $u_a-u_b+\varepsilon$ is a solution to the uniformly elliptic equation with zero boundary value. Hence $u_a-u_b+\varepsilon=0$ in $\{u_a-u_b< -\varepsilon\}$, which is a contradiction.

{\it Case 2b.  $\{u_a-u_b< -\varepsilon\}$ is unbounded.} In this case, as in Case 1b we know Proposition \ref{Prop: unbounded estimate} implies $u_b-u_a-\varepsilon$ is unbouded, which is a contradiction.

Case 2a and Case 2b imply  $u_a-u_b\geq-\varepsilon$ for any $\varepsilon>0$ and hence $u_a\geq u_b$. A similar argument also implies  $u_b\geq u_a$. So $u_a=u_b$ when $a=b$.
\end{proof}

As an immediate consequence of Proposition \ref{halfconst} and Proposition \ref{abmono}.
\begin{corollary}\label{halfconstcoro}
If $\Omega= \mathbf{R}^n_+$, then for any solution $u$ of \eqref{Eq: MSE}-\eqref{Eq: boundary value}, there is a unique constant $c$, such that $u=u_c$, where $u_c$ comes from Proposition \ref{existhalf}.
\end{corollary}

We rewrite the second part of Theorem \ref{Thm: main} to the following proposition.

\begin{proposition}\label{homeo}
For any $c\in\mathbf R$, let $u_c$ be the solution constructed in Proposition \ref{existhalf}. Then the map
$$
\Phi:\mathbf R^n_+\times \mathbf R\to \mathbf R^n_+\times \mathbf R,\, (x,c)\mapsto \left(x,u_c(x)\right)
$$
is a homeomorphism.
\end{proposition}
\begin{proof}
We first show the continuity of $\Phi$, and it suffices to prove $u_c(x)$ is continuous with respect to $(x, c)$. Take any sequence $(p_k, c_k)\to (p, c)$ as $k\to \infty$, and by definition we know for any $k$, $u_{c_k}$ is solution of \eqref{Eq: MSE}-\eqref{Eq: boundary value} with estimate
$$|u_{c_k}-l-c_kx_n|\leq \|\phi\|_{C^0(\mathbf{R}^n)}.$$
By interior gradient estimate for the minimal surface equation and Lemma \ref{uni-continuous}, up to a subsequence the function $u_{c_k}$ converges to a limit function $u'$ in $C^\infty_{loc}(\mathbf R^n_+)$ and $C^0_{loc}(\overline{\mathbf {R}^n_+})$ satisfying \eqref{Eq: MSE}-\eqref{Eq: boundary value}. Also, $u'$ satisfies $$|u'-l-cx_n|\leq \|\phi\|_{C^0(\mathbf{R}^n)}.$$ By Proposition \ref{abmono}, $u'=u_c$ and hence $u_{c_k}$ converges to $u_c$ in $C^0_{loc}(\overline{\mathbf {R}^n_+})$. This implies $u_{c_k}(p_k)\to u_c(p)$ as $k\to\infty$.

Next,  the bijectivity of $\Phi$ comes from Proposition \ref{existhalf}, Proposition \ref{halfconst}, Proposition \ref{abmono} and the continuity of $\Phi$. 

Finally let us prove $\Phi$ is a closed map, that is, $\Phi(A)$ is closed in $\R_+^n\times\R$ for any closed set $A\subset\R_+^n\times\R.$  
Suppose now $\{(p_k,c_k)\}_{k\geq1}\subset A$ is a sequence such that $$\Phi(p_k,c_k)=(p_k,u_{c_k}(p_k))\to(p_0, q)\mbox{ in } \R_+^n\times\R,\mbox{ as }k\to\infty,$$
and here we write $p_k=(x_1(p_k),\cdots,x_n(p_k))$ and $p_0=(x_1(p_0),\cdots,x_n(p_0)).$ Then, $$x_n(p_k)\to x_n(p_0)>0\mbox{ as }k\to\infty.$$ By Proposition \ref{halfconst}, we deduce that 
$$
\left|\frac{u_{c_k}(p_k)-l(p_k)}{x_n(p_k)}-c_k\right|\leq\frac{\|\phi\|_{C^0(\mathbf{R}^n)}}{x_n(p_k)}$$
for any $k\geq1.$ Thus, $\{c_k\}$ is bounded in $\R.$ Up to a subsequence, we may assume $\{c_k\}$ converges to $c_0$ as $k\to\infty.$ Then, $(x_0,c_0)\in A$ and it follows from the continuity of $\Phi$ that $$q=\lim_{k\to\infty}u_{c_k}(p_k)=u_{c_0}(p_0).$$ Hence, $\Phi(A)$ is closed in $\mathbf{R}_+^n\times\mathbf{R}.$
\end{proof}

The graphs of functions $u_c$ actually form a differential foliation of $\mathbf{R}_+^n\times\mathbf{R}$ if the function $\phi$ has better regularity. For our purpose, let us change the parametrization for the family of functions $u_c$. Note that the restriction map
$$
\Phi_{x_0}:\mathbf R\to \mathbf R,\, c\mapsto u_c(x_0),
$$
is also a homeomorphism for any point $x_0$ in $\mathbf R^n_+$. 

Fix a point $x_0$ in $\mathbf R^n_+$. We define
$$
\bar u_t(x)=u_{c(t)}(x) \mbox{ with } c(t)=\Phi_{x_0}^{-1}(t)
$$
and consider the corresponding map
$$
\bar \Phi:\mathbf R^n_+\times \mathbf R\to \mathbf R^n_+\times \mathbf R,\, (x,t)\mapsto \left(x,\bar u_t(x)\right).
$$

We have the following
\begin{proposition}\label{halfc1}
If $\phi$ is $C^1$ on $\partial\mathbf R^n_+$ with $\|\phi\|_{C^1(\partial \mathbf R^n_+)}<+\infty$, then the map $\bar\Phi$ is a $C^1$-diffeomorphism.
\end{proposition}
\begin{proof}
First notice that we can improve estimates for functions $u_c$ from the $C^1$-bound of $\phi$. From the gradient estimate for the minimal surface equation (see \cite{GT, Han, JWZ2019}), for any $c_*>0$ there is a universal constant $C=C(c_*,\|\phi\|_{C^1(\partial\mathbf{R}^n_+)},n)$ such that
\begin{equation}\label{Eq: gradient}
|\nabla u_c(x)|\leq C,\mbox{ for all } c \in [-c_*,c_*] \mbox{ and } x\in \overline{\mathbf R^n_+}.
\end{equation}
On the other hand, given any compact subset $K$ in $\mathbf R^n_+$ and any nonnegative integer $k$ there is a universal constant $C_k=C_k(K,c_*,n)$ such that
\begin{equation}\label{Eq: interior estimate}
|\nabla^ku_c(x)|\leq C_k,\mbox{ for all } c \in [-c_*,c_*]\mbox{ and }x\in K.
\end{equation}

The rest proof will be divided into three steps.

{\it Step 1. The functions $\bar u_t(x)$ is differentiable with respect to $t$.} We are going to show
$$
\frac{\partial}{\partial t}\bar u_t(x)=\bar v_t(x),
$$
where $\bar v_t$ is the unique solution to the equation
\begin{equation}\label{Eq: equation bar v t}
\partial_i(\bar a_{ij,t}\partial_j\bar v_t)=0 \mbox{ in } \mathbf R^n_+
\end{equation}
with the condition 
\begin{equation}\label{Eq: condition bar v t}
\bar{v}_t=0 \mbox{ on } \partial\mathbf R^n_+,\mbox{ and }\bar v_t(x_0)=1,
\end{equation}
where
$$
\bar a_{ij,t}=\delta_{ij}-\frac{\partial_i\bar u_{t}\partial_j\bar u_{t}}{(1+|\nabla\bar u_t|^2)^{\frac{3}{2}}}.
$$
The uniqueness of $\bar v_t$ follows from the gradient estimate \eqref{Eq: gradient} and Theorem \ref{Thm: positive solution in cone}.
For any real number $\tau$, let us take 
$$
\bar v_{t,\tau}=\tau^{-1}\left(\bar u_{t+\tau}-\bar u_t\right).
$$
It suffices to show that for any sequence $\tau_k$ such that $\tau_k\to 0$ as $k\to+\infty$ the functions $\bar v_{t,\tau_k}$ converges $\bar v_t$. The idea is to use the uniqueness of $\bar v_t$. Clearly, by Proposition \ref{abmono} each $\bar v_{t,\tau_k}$ is a positive function on $\mathbf R^n_+$ with value 1 at point $x_0$, which vanishes on the boundary $\partial\mathbf R^n_+$. It also satisfies the equation
\begin{equation}\label{Eq: bar v t tau}
\partial_i(\bar a_{ij,t,\tau_k}\partial_j\bar v_{t,\tau_k})=0 \mbox{ in } \mathbf R^n_+,
\end{equation}
where
$$
\bar a_{ij,t,\tau_k}=\delta_{ij}-\int_0^1\frac{\partial_i\bar w_{s,\tau_k}\partial_j\bar w_{s,\tau_k}}{(1+|\nabla\bar w_{s,\tau_k}|^2)^{\frac{3}{2}}}\,\mathrm ds
$$
with
$$
\bar w_{s,\tau_k}=(1-s)\bar u_t+s\bar u_{t+\tau_k}.
$$
From the interior estimate \eqref{Eq: interior estimate} the coefficients $\bar a_{ij,t,\tau_k}$ converges smoothly to $\bar a_{ij,t}$ in any compact subset of $\mathbf R^n_+$ up to a subsequence. On the other hand, we also have good estimates for functions $\bar v_{t,\tau_k}$. As a beginning, we point out that they have locally uniform $C^0$-bounds. To see this, we take $\rho>0$ large enough such that the semi-ball $B^+_\rho=B_\rho\cap \mathbf R^n_+$ contains the point $x_0$, and then construct a solution $\bar w_{\tau_k}$ of equation \eqref{Eq: bar v t tau} in $B_\rho^+$ by prescribing a Dirichlet boundary value $\psi\in C^0(\partial B_\rho^+)$, which is positive on $\partial B_\rho\cap \mathbf R^n_+$ and vanishes on $\partial\mathbf R^n_+$. It is not difficult to see that $\{\bar w_{\tau_k}\}_k$ is compact in $C^0(\overline{B_\rho^+})$ and so  $$\bar w_{\tau_k}(x_0)>\epsilon_0,$$ where $\epsilon_0$ is a positive constant independent of $k$. Then the boundary Harnack inequality (see Theorem \ref{Thm: boundary harnack}) combined with the maximum principle yields that
$$
\bar v_{t,\tau_k}\leq C'\epsilon_0^{-1}\|\psi\|_{C^0} \mbox{ in } B_{\rho/2}^+,
$$
where $C'$ is a positive constant independent of $k$.
From the standard elliptic PDE theory, the functions $\bar v_{t,\tau_k}$ must have locally uniform up-to-boundary H\"older estimate and locally uniform $C^l$-estimates in $\mathbf R^n_+$, so it converges to a limit function $\bar v_t'$ up to a subsequence in $C^\infty_{loc}(\mathbf R^n_+)\cap C^0_{loc}(\overline{\mathbf R^n_+})$, which solves \eqref{Eq: equation bar v t}-\eqref{Eq: condition bar v t}. Since \eqref{Eq: equation bar v t}-\eqref{Eq: condition bar v t} has a unique solution, it implies $\bar v_t'=\bar v_t$, and we obtain the desired consequence.

{\it Step 2. All partial derivatives of the function $\bar U(x,t):=\bar u_{t}(x)$ are continuous with respect to $(x,t)$.} First we deal with $\bar v_{t}(x)$, the partial derivative of $\bar U(x,t)$ with respect to $t$. It suffices to show that $\bar v_{t'}$ converges to $\bar v_t$ in $C^\infty_{loc}(\mathbf R^n_+)$ as $t'\to t$ and
the proof is almost identical to that in Step 1. The only modification is that the functions $\bar v_{t,\tau_k}$ need to be replaced by the functions $\bar v_{t_k}$ with $t_k\to t$ as $k\to+\infty$, where $\bar v_{t_k}$ is the unique solution of \eqref{Eq: equation bar v t}-\eqref{Eq: condition bar v t} (after $t$ is replaced by $t_k$). The argument in Step 1 goes smoothly without any difficulty in this case.

It remains to show that $\partial_i\bar u_t$ is continuous with respect to $(x,t)$. With the same idea, we would like to prove that $\bar u_{t'}$ converges to $\bar u_t$ in $C^\infty_{loc}(\mathbf R^n_+)$ as $t'\to t$. Take any sequence $t_k\to t$ as $k\to+\infty$. Recall that $\bar u_t$ is exactly the function $u_{c(t)}$, where $c(t)=\Phi_{x_0}^{-1}(t)$ is continuous with respect to $t$. Similar as the proof of the continuity part in Proposition \ref{homeo}, from Proposition \ref{existhalf} and Lemma \ref{uni-continuous}, up to a subsequence the function $u_{c(t_k)}$ converges to a limit function $u_{c(t)}$ in $C^\infty_{loc}(\mathbf R^n_+)\cap C^0_{loc}(\overline{\mathbf {R}^n_+})$ satisfying \eqref{Eq: MSE}-\eqref{Eq: boundary value}. %Moreover, $u'$ satisfies the following $C^0$-estimate
%$$
%\|u'-l-c(t)x_n\|_{C^0(\mathbf R^n_+)}\leq \|\phi\|_{C^0(\partial\mathbf R^n_+)}.
%$$
%By Proposition \ref{abmono} it follows that $u'=u_{c(t)}$. 
%Therefore, we have shown that $u_{c(t')}$ converges to $u_{c(t)}$ in $C^\infty_{loc}(\mathbf R^n_+)$ as $t'\to t$.

{\it Step 3. $\bar \Phi$ is a $C^1$-diffeomorphism.} From the definition we see
$$
\bar \Phi(x,t)=\Phi\left(x,\Phi^{-1}_{x_0}(t)\right).
$$
This yields that $\bar \Phi$ is a homeomorphism. According to the inverse function theory, all we need to show is that the Jacobian of the map $\bar \Phi$ has non-zero determinant. It is not difficult to see
\[
J_{\bar \Phi}=\left(
\begin{array}{cc}
I_{n\times n}&0\\
\ast& \bar v_t(x)
\end{array}\right).
\]
From Step 1 we know $\bar v_t(x)$ is positive in $\mathbf R^n_+$, and hence the determinant of $J_{\bar\Phi}$ is non-zero. So we complete the proof.
\end{proof}

\begin{remark}
From our proof, if $\phi$ is $C^k$ on $\partial\R^n_+$ with bounded $\|\phi\|_{C^k(\partial\R^n_+)}$, then $\bar\Phi$ is a $C^k$-diffeomorphism.
\end{remark}

\appendix
\section{Some preliminary results}\label{Sec: A}
{
\subsection{Boundary Harnack inequality}\label{Subsec: uniqueness}
First we will recall the boundary Harnack inequality from \cite{DSS20}.
Let $x=(x',x_n)$ and $B_1'$ be the unit ball in $\mathbf R^{n-1}$. Let $g:B_1'\to \mathbf R$ be a Lipschitz function with Lipschitz norm $L$ and $g(0)=0$. Define the graph
$$
\Gamma=\{(x',x_n)\in\mathbf R^n:x_n=g(x'),\,x'\in B_1'\}
$$
and the height function
$$
h_{\Gamma}:B_1'\times \mathbf R\to\mathbf R,\, (x',x_n)\mapsto x_n-g(x').
$$
For convenience, we let
$$
\mathcal C_{r,\rho}=\{x=(x',x_n)\in B_1'\times \mathbf R: |x'|\leq r,\,0<h_\Gamma(x)<\rho\}
$$
and $\mathcal C_{r}=\mathcal C_{r,r}$. 
\begin{figure}[htbp]
\centering
\includegraphics[width=6cm]{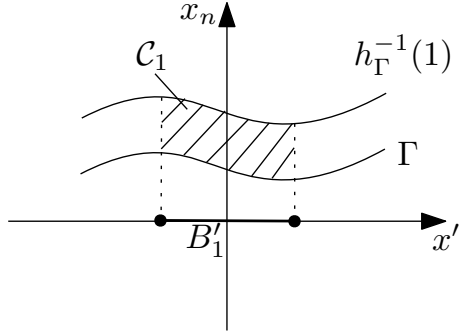}
\caption{Domain $\mathcal C_1$ above Lipschitz graph $\Gamma$}
\label{Fig: 2}
\end{figure}

In the following, we consider the following uniformly elliptic operator of divergence form
$$
\mathcal Lu:=\Div(A(x)\nabla u),
$$
with $A(x)$ is bounded and measurable satisfying
$$
\lambda I\leq A(x)\leq \lambda^{-1} I, \mbox{ for some constant } \lambda\in(0, 1).
$$

The boundary Harnack inequality can be stated as the following.
\begin{theorem}\label{Thm: boundary harnack}
If $u_1$ and $u_2$ are solutions of the equation
$$
\mathcal Lu=0 \mbox{ in } \mathcal C_1
$$
satisfying
\begin{itemize}
\item $u_1$ and $u_2$ vanish on $\Gamma$,
\item $u_1$ and $u_2$ are positive in $\mathcal C_1$,
\item and $u_1(0,\frac{1}{2})=u_2(0,\frac{1}{2})=1$,
\end{itemize}
then there is a universal constant $C=C(n,\lambda, L)$ such that
$$
C^{-1}\leq \frac{u_1}{u_2}\leq C \mbox{ in } \mathcal C_{1/2}.
$$
\end{theorem}
\subsection{Positive solutions of uniformly elliptic linear equations in cones}
In the following, the operator $\mathcal L$ is assumed to satisfy the same requirements in the previous subsection. {The symbol $\mathcal C$ means an infinite Lipschitz cone now.}
\subsubsection{Existence and uniqueness of the solution}
\begin{theorem}\label{Thm: positive solution in cone}
The equation $\mathcal Lu=0$ in $\mathcal C$ with $u=0$ on $\partial\mathcal C$ and $u>0$ in $\mathcal C$ has a unique solution in $W^{2,p}_{loc}(\mathcal C)\cap C^0(\bar{\mathcal C})$ up to a scaling.
\end{theorem}

We point out that Theorem \ref{Thm: positive solution in cone} is the main theorem proved in \cite{LN85} by Landis and Nadirashvili. For reader's convenience, here we provide a different proof inspired from \cite{BWZ16} based on the boundary Harnack inequality.

\begin{proof}[Proof of existence]
Let $r_n$ be a sequence of positive real numbers with $r_n\to+\infty $ as $n\to\infty$. We let $B_n$ be the ball centered at the pole of $\mathcal C$ with radius $r_n$ and $\mathcal C_n$ the intersection $\mathcal C\cap B_n$. It is easy to construct a continuous function $\phi_n$ on $\partial\mathcal C_n$ such that $\phi$ vanishes on $\partial\mathcal C\cap B_n$ and $\phi$ is positive on $\mathcal C\cap\partial B_n$. Clearly, we can find a solution $u_n\in W^{2,p}(\mathcal C_n)\cap C^0(\bar{\mathcal C_n})$ of the equation $\mathcal Lu_n=0$ in $\mathcal C_n$ with the Dirichlet boundary value $\phi_n$ (see \cite{GT} for instance). The maximum principle then yields that $u_n$ is positive in $\mathcal C_n$. Fix a point $P$ in the intersection of all $\mathcal C_n$. Up to a scaling we can normalize the function $u_n$ to satisfies $u_n(P)=1$. From the interior Harnack inequality, the boundary Harnack inequality (Theorem \ref{Thm: boundary harnack}) and the $W^{2, p}$-estimate, $u_n$ converges to a nonnegative function $u$ satisfying $\mathcal Lu=0$ in $\mathcal C$ with zero boundary value on $\partial\mathcal C$.
\end{proof}

\begin{proof}[Proof of uniqueness]
Fix a ray $\gamma$ in $\mathcal C$ starting from the pole of $\mathcal C$. First we are going to show the following property: if $u$ and $v$ are two solutions satisfying the hypothesis in Theorem \ref{Thm: positive solution in cone} such that $u$ equals to $v$ at the point $\gamma\cap\partial B_r$ for some $r>0$, then there is a universal constant $C_0=C_0(n,\lambda,\mathcal C)$ such that
$$
C_0^{-1}\leq \frac{u}{v}\leq C_0 \mbox{ in } \mathcal C.
$$

To prove this, we first notice that if $u$ equals $v$ at the point $\gamma\cap\partial B_r$, then there is a universal constant $C_1=C_1(n,\lambda,\mathcal C)$ such that
$$
C_1^{-1}\leq \frac{u}{v}\leq C_1 \mbox{ on }\partial B_r\cap\mathcal C.
$$
By the maximum principle, we conclude that $C_1^{-1}\leq u/v\leq C_1$ holds for any point in $\mathcal C\cap B_r$. On the other hand, suppose we have the inequality $$C_1^{-1}\leq \frac{u}{v}\leq C_1$$ at some point in $\mathcal C\cap B_s$, then
\begin{equation}\label{outer}C_1^{-2}\leq \frac{u(P)}{v(P)}\leq C_1^2,\end{equation}  where $P=\gamma\cap \partial B_s$. Otherwise, if $u/v>C_1^2$ at the point $P$, then there is a constant $C(P)$ such that $C(P)>C_1^2$ and $u(P)=C(P)v(P)$. By applying previous argument to $u$ and $C(P)v$, we have $$u\geq C^{-1}_1 C(P)v>C_1 v$$ in $\mathcal{C}\cap B_s$, which is a contradiction. The left hand side inequality of \eqref{outer} could be proved in the same way.  Clearly, \eqref{outer} holds for all large $s$ and it provides a control of the ratio $u/v$ at infinity. Applying the boundary Harnack inequality and the maximum principle once again, finally we arrive at
$$
C_0^{-1}\leq \frac{u}{v}\leq C_0\mbox{ in } \mathcal C,$$ where $C_0=C_1^3.$

Now the uniqueness is almost direct. Let $u$ be the solution contructed above and $v$ be any other solution. From previous discussion, we see that there is a positive constant $\epsilon$ such that $\epsilon u\leq v\leq \epsilon^{-1} u$. Define
$$
\epsilon^*:=\sup\left\{\epsilon>0:v>\epsilon u\right\}
$$
and we consider the function $w:=v-\epsilon^* u$. Clearly, $w$ is a nonnegative solution of the equation $\mathcal Lw=0$ in $\mathcal C$ with zero boundary value on $\partial \mathcal C$. The Harnack inequality yields that $w$ is either the zero function or a positive function in $\mathcal C$. In the latter case, we can find another positive constant $\epsilon'$ such that $w\geq \epsilon' u$ and so $v\geq(\epsilon^*+\epsilon')u$. This contradicts to the definition of $\epsilon^*$ and we complete the proof.
\end{proof}
As a special case of Theorem \ref{Thm: positive solution in cone}, we see
\begin{corollary}\label{Cor: harmonic in cone}
If $\mathcal L=\Delta$ is the Laplace operator and the cone $\mathcal C$ has the form of
$
\mathcal C=\{tx:t>0,\,x\in S\}
$
for a smooth domain $S$ in $\mathbf S^{n-1}$, then any harmonic function in $\mathcal C$ with $u=0$ on $\partial\mathcal C$ and $u>0$ in $\mathcal C$ has the form $u=cr^\beta \phi_1$, where $c$ is a positive constant, $\phi_1$ is the first eigenfunction of $S$ with corresponding first eigenvalue $\lambda_1$ and
$$
\beta=\frac{-(n-2)+\sqrt{(n-2)^2+4\lambda_1}}{2}.
$$
\end{corollary}

\subsubsection{H\"older decay or blow up alternative of the solution from the Harnack inequality}

\begin{lemma}\label{Lem: a drop inside}
{Let $\mathcal C_\rho=\mathcal C\cap B_\rho$ and let $\mathcal A_r$} be the annulus $\mathcal A_r:=\mathcal C_{4r}-\bar {\mathcal C}_r$ for any $r>0$. If $u\in W^{1,2}(\mathcal A_r)\cap C^0(\bar{\mathcal A_r})$ is a solution of $\mathcal Lu=0$ in $\mathcal A_r$ such that $u\leq 0$ on the side boundary $\partial \mathcal C\cap(B_{4r}-\bar B_r)$, then we have
$$
\sup_{\partial B_{2r}\cap \mathcal C}u^+\leq (1-\delta)\sup_{\partial \mathcal A_r} u^+
$$
for some universal constant $\delta=\delta(n,\lambda,\mathcal C)$, where $u^+$ is the positive part of $u$.
\end{lemma}
\begin{proof}
If $u$ is non-positive on $\partial \mathcal A_r$, then it follows from the maximum principle that both sides of above inequality equal to zero. So we just need to deal with the case when $u$ is positive somewhere on $\partial\mathcal A_r$.  Without no loss of generality, we can assume $u=0$ on the side boundary $\partial \mathcal C\cap(B_{4r}-\bar B_r)$ and $0\leq u\leq 1$ on $\partial \mathcal A_r$, otherwise we can introduce $u'$ and here $u'$ is the solution to the same elliptic equation in $\mathcal A_r$ with the Dirichlet boundary value $$\left(\sup_{\mathcal \partial \mathcal{A}_r}u\right)^{-1}\max\left\{0, u|_{\mathcal \partial \mathcal{A}_r}\right\}.$$ 
%Since we are going to derive an upper bound for $u$, we can further assume $u$ is nonnegative, otherwise we can add a positive constant to $u$ and handle the scaling of the new function.
%by the maximum principle we can just increase the boundary value and assume $u=u^+$ on $\partial \mathcal A_r$. 
Based on these assumptions, it is easy to see $0\leq u\leq 1$ by the maximum principle and $u=u^+$ on $\partial \mathcal A_r$. Without loss of generality, we only deal with the case when $r=1$.

Let us consider the function $v=1-u$. Note the $L^\infty$-norm of $v$ is bounded by 1, we conclude that $v$ is bounded below by $\frac{1}{2}$ in some neighborhood of $\partial\mathcal C\cap \partial B_2$ by the boundary H\"older estimate (see \cite[Theorem 3.6]{LZLH}). Combined with the Harnack inequality, $v$ must be bounded below by a positive constant $\delta$ on $\partial B_2\cap\mathcal C$. This yields $u\leq 1-\delta$ on $\partial B_2\cap\mathcal C$ and we complete the proof.
\end{proof}

Based on this lemma, we have the following alternative.
\begin{lemma}\label{Lem: alternative}
Let $u\in W^{2,p}_{loc}(\mathcal C)\cap C^0(\bar{\mathcal C})$ be a solution of the equation $\mathcal Lu=0$ in $\mathcal C$ with $u=0$ on $\partial\mathcal C-B_1$ and $u>0$ in $\mathcal C-B_1$. Then the function
$$
\mathrm{osc}(r)=\max_{\mathcal C\cap \partial B_r} u
$$
has a limit as $r\to+\infty$, and the limit is either $+\infty$ or $0$. Moreover, either $\Osc(r)\geq cr^\beta$ or $\Osc(r)\leq cr^{-\beta}$ holds for some positive constants $c$ and $\beta$. 
\end{lemma}
\begin{proof}
It is clear that $u$ cannot be a constant function. Then it follows from the maximum principle that the function $\Osc(r)$ cannot have a local maximum, and so it must be monotone when $r$ is large enough. Hence the function $\Osc(r)$ has a limit when $r$ tends to $+\infty$.

To show the alternative for the limit of $\Osc(r)$, we just need to show $\Osc(r)\to 0$ as $r\to+\infty$ under the assumption that $\Osc(r)$ is bounded from above by a positive constant $C$. Fix $r_0$ to be a large positive constant. Since $u\leq C$ holds on $\partial B_{r}\cap\mathcal C$ for all $r\geq r_0$, we obtain by Lemma \ref{Lem: a drop inside} that $u\leq (1-\delta)C$ on $\partial B_{r}\cap\mathcal C$ for all $r\geq 2r_0$. Through iteration we conclude that $u\leq (1-\delta)^kC$ on $\partial B_{r}\cap\mathcal C$ for all $r\geq 2^kr_0$, which yields $\Osc(r)\to 0$ as $r\to +\infty$. 
%Finally we estimate the order of $\Osc(r)$. 
From above discussion, it is not difficult to deduce the inequality
$$
\Osc(r)\leq C\left(\frac{r}{r_0}\right)^{\log_2(1-\delta)},\mbox{ for all }r\geq r_0.
$$
This yields $\Osc(r)\leq cr^{-\beta}$ for some positive constants $c$ and $\beta$ in the case when $\Osc(r)\to 0$.  

Now let us deal with the case when $\Osc(r)\to +\infty$. At this time, $\Osc(r)$ must be monotone increase when $r\geq r_0$ for some positive constant $r_0$. From Lemma \ref{Lem: a drop inside} we have
$$
\Osc\left(4^kr_0\right)\geq (1-\delta)^{-k}\Osc(r_0).
$$
This can be used to deduce
$$
\Osc(r)\geq \Osc(r_0)\left(\frac{r}{r_0}\right)^{-\log_4(1-\delta)}, \mbox{ for all } r\geq r_0.
$$
This yields $\Osc(r)\geq cr^{\beta}$ for some positive constants $c$ and $\beta$ and we complete the proof.
\end{proof}
Now it is clear that
\begin{corollary}\label{Cor: unbounded}
Let $u\in W^{2,p}_{loc}(\mathcal C)\cap C^0(\bar{\mathcal C})$ be a solution of the equation $\mathcal Lu=0$ in $\mathcal C$ with $u=0$ on $\partial\mathcal C$ and $u>0$ in $\mathcal C$. Then $u$ is unbounded.
\end{corollary}
\begin{proof}
Since $u$ cannot be the zero function, it follows from the maximum principle and the alternative that $\Osc(r)\to +\infty$ as $r\to +\infty$.
\end{proof}
}

\section*{Acknowledgments}
G. Jiang is supported by China Postdoctoral Science Foundation 2021TQ0046. Z. Wang is supported by the Special Research Assistant Project at Chinese Academy of Sciences. J. Zhu is supported by China Postdoctoral Science Foundation BX2021013.
\bibliography{bib}
\bibliographystyle{amsplain}
\end{document}